\documentclass[12pt]{scrartcl}
\usepackage[T1]{fontenc}
\usepackage[utf8]{inputenc}
\usepackage{booktabs}
\usepackage{multirow}

\usepackage{todonotes,subfig,wrapfig}

\usepackage{amssymb}
\usepackage{amsmath}
\usepackage{amsfonts}
\usepackage{amsthm}

\usepackage{parskip}

\newtheorem{theorem}{Theorem}
\usepackage{hyperref}

\newcommand{\rr}{\mathbb{R}}
\newcommand{\jump}[1]{{[\![#1]\!]}}
\newcommand{\enorm}[1]{|\!|\!|#1|\!|\!|}

\usepackage[scientific-notation=true]{siunitx}
\usepackage{graphicx}
\graphicspath{{./images/}}

\title{Higher order unfitted FEM for Stokes interface problems}

\author{
Philip Lederer\thanks{Institute for Analysis and Scientific Computing, TU Wien,
 Wiedner Hauptstr. 8-10, 1040 Wien, Austria}
\and
Carl-Martin Pfeiler\footnotemark[1] 
\and
Christoph Wintersteiger\footnotemark[1] 
\and
Christoph Lehrenfeld\thanks{Institut f\"ur Numerische und Angewandte Mathematik, WWU M\"unster, Einsteinstr. 62, 48149 M\"unster, Germany, email: \texttt{christoph.lehrenfeld@gmail.com}}
}

\begin{document}
\maketitle                   


\begin{abstract}
We consider the discretization of a stationary Stokes interface problem in a velocity-pressure formulation. 
The interface is described implicitly as the zero level of a scalar function as it is common in level set based methods. Hence, the interface is not aligned with the mesh.
An unfitted finite element discretization based on a Taylor-Hood velocity-pressure pair and an XFEM (or CutFEM) modification is used for the approximation of the solution. This allows for the accurate approximation of solutions which have strong or weak discontinuities across interfaces which are not aligned with the mesh. To arrive at a consistent, stable and accurate formulation we require several additional techniques. 
First, a Nitsche-type formulation is used to implement interface conditions in a weak sense. 
Secondly, we use the ghost penalty stabilization to obtain an inf-sup stable variational formulation. 
Finally, for the highly accurate approximation of the implicitly described geometry, we use a combination of a piecewise linear interface reconstruction and a parametric mapping of the underlying mesh. 
We introduce the method and discuss results of numerical examples.
\end{abstract}

\section{Introduction}
We consider the two-phase Stokes problem on the open domain $\Omega \subset \rr^d,~d=2,3$ with two disjoint subdomains $\Omega_1$, $\Omega_2$ with $\Omega_1\cap\Omega_2=\emptyset$, $\overline{\Omega}_1 \cap \overline{\Omega}_2\! =\! \Gamma$, $\Omega_1 \cup \Omega_2 \cup \Gamma \!=\! \Omega$. We assume that one phase is completely surrounded by the other, i.e. $\partial \Omega \cap \Gamma \!=\! \emptyset$.
\begin{subequations}\label{eq:modprob}
\begin{eqnarray} 
  - \mathrm{div}(\mu_i D(\mathbf{u})) + \nabla p & = \rho_i \mathbf{g} & \text{ and } ~
 \mathrm{div}(\mathbf{u}) = 0 \quad \text{ in } \Omega_i, i=1,2, \label{eq:inner} \\
 \jump{\mathbf{u}} & = 0 & \text{ and } ~
\jump{\sigma(\mathbf{u},p) \cdot \mathbf{n}} = \mathbf{f} ~ \text{ on } \Gamma ~\text{ and } \mathbf{u} = \mathbf{u}_D \text{ on } \partial \Omega. \label{eq:interface}
\end{eqnarray}
\end{subequations}
Here, $\rho$ is the domainwise constant density, $\mu$ the domainwise constant viscosity, $\mathbf{g} \in [L^2(\Omega)]^d$ the gravitational force and $\mathbf{f} \in [L^2(\Gamma)]^d$ the surface tension force. $\jump{\cdot}$ is the usual jump operator across the interface, $\jump{v} := v|_{\Omega_1} - v|_{\Omega_2}$, $D(\mathbf{u})$ denotes the symmetric gradient $D(\mathbf{u}) := \nabla \mathbf{u} + \nabla \mathbf{u}^\top$ and $\sigma(\mathbf{u},p) = -\mu D(\mathbf{u}) + p \ \mathbf{I}$ is the stress tensor. 
We assume that the solution has the regularity $\mathbf{u} \in [H^1(\Omega)]^d \cap [H^3(\Omega_1\cup\Omega_2)]^d$ and $p \in L^2_0(\Omega) \cap H^2(\Omega_1\cup\Omega_2)$ with $L^2_0(\Omega) = \{ v \in L^2(\Omega) | \int_{\Omega} v \, dx = 0\}$. The interface is described only implicitly as the zero level of a (sufficiently smooth) scalar function, i.e. $\Gamma = \{ \phi = 0 \}$, but the computational mesh is not aligned to $\Gamma$, i.e. we consider a discretization in an ``unfitted'' setting. 

For the discretization, different challenges arise due to weak (velocity) and strong (pressure) discontinuities across $\Gamma$ and the approximation of the implicitly described geometries.
The major result of this contribution is the presentation of a new unfitted finite element method for the Stokes interface problem with order-optimal error bounds.
The method is presented in section \ref{sec:disc} and consists of a combination of enriched approximation spaces close to the interface (sec. \ref{sec:velpre}), Nitsche's method to implement the interface conditions in a weak sense (sec. \ref{sec:nitsche}), a ghost penalty stabilization to ensure inf-sup stability 
(sec. \ref{sec:ghostpen}) and a proper approach for numerical integration on level set domains
(sec. \ref{sec:geomapprox}). In section \ref{sec:numex} numerical examples are shown and discussed.

\section{Discretization spaces and variational formulation}\label{sec:disc}
\subsection{Choice of the velocity-pressure pair and the basic variational formulation}\label{sec:velpre}
Let $\mathcal{T}_h$ be a simplex triangulation of the domain $\Omega$ which is not necessarily aligned to $\Gamma$. 
As a starting point for the discretization we consider the famous Taylor-Hood velocity pressure space $\mathbf{V}_h \times Q_h$ which is known to be LBB-stable with
\begin{align*}
 \mathbf{V}_h & := \{ \mathbf{v} \in [C(\Omega)]^d |~ \mathbf{v}|_T \in [\mathcal{P}^2(T)]^d, T \in \mathcal{T}_h \},\\
 Q_h & := \{ v \in C(\Omega) |~ v|_T \in \mathcal{P}^1(T), T \in \mathcal{T}_h \},
\end{align*}
where $\mathcal{P}^k(T)$ is the space of polynomials up to degree $k\in\{1,2\}$ on $T \in \mathcal{T}_h$. 
Due to the fact that the velocity can have weak discontinuities (kinks) and the pressure can have discontinuities (jumps) across the unfitted interface, this velocity-pressure pair offers only a very poor approximation quality to the solution $(\mathbf{u},p)$ of \eqref{eq:modprob}. There hold the sharp (w.r.t. $h$) estimates
\begin{subequations}
\begin{equation} \label{eq:suboptest}
\begin{split}\inf_{\mathbf{v}_h \in \mathbf{V}_h} \Vert \mathbf{v}_h - \mathbf{u} \Vert_{H^1(\Omega_1\cup\Omega_2)} & \lesssim h^{\frac12} \Vert \mathbf{u} \Vert_{H^2(\Omega_1\cup\Omega_2)}, \\
\inf_{q_h \in Q_h} \Vert q_h - p \Vert_{L^2(\Omega)} & \lesssim h^{\frac12} \Vert p \Vert_{H^1(\Omega_1\cup\Omega_2)}.
\end{split}
\end{equation}
To deal with unfitted discontinuities standard finite element spaces are adjusted in the unfitted finite element method which is also known under the names CutFEM \cite{burman2014cutfem} or XFEM \cite{fries2010extended} in the literature. We use the finite element spaces
$\mathbf{V}_h^\Gamma := \mathbf{V}_h|_{\Omega_1} \oplus \mathbf{V}_h|_{\Omega_2}$
and 
$Q_h^\Gamma := Q_h|_{\Omega_1} \oplus Q_h|_{\Omega_2}$
as they are also considered in (among others) \cite{hansbo2002unfitted,grossreusken07,reusken08}. 
This gives rise to the estimates
\begin{equation}\label{eq:secorderbound}
\begin{split}
\inf_{\mathbf{v}_h \in \mathbf{V}_h^\Gamma} \Vert \mathbf{v}_h - \mathbf{u} \Vert_{H^1(\Omega_1\cup\Omega_2)} & \lesssim h^{2} \Vert \mathbf{u} \Vert_{H^3(\Omega_1\cup\Omega_2)}, \\
\inf_{q_h \in Q_h^\Gamma} \Vert q_h - p \Vert_{L^2(\Omega)} & \lesssim h^{2} \Vert p \Vert_{H^2(\Omega_1\cup\Omega_2)}.
\end{split}
\end{equation}
\end{subequations}
The resulting velocity-pressure pair $\mathbf{V}_h^\Gamma \times Q_h^\Gamma$ is suitable to approximate solutions with (strong and weak) discontinuities across the interface, but it is nonconforming in the velocities, $\mathbf{V}_h^\Gamma \not\subset [H^1(\Omega)]^d$.
Further, we note that the LBB-stability of the underlying velocity-pressure pair $\mathbf{V}_h \times Q_h$
is not inherited by $\mathbf{V}_h^\Gamma \times Q_h^\Gamma$.\\[1ex]

With bilinear forms $N(\cdot,\cdot)$ and $J(\cdot,\cdot)$ to be introduced in subsections \ref{sec:nitsche} and \ref{sec:ghostpen}, which are responsible for dealing with the nonconformity of $\mathbf{V}_h^\Gamma$, the interface conditions \eqref{eq:interface} and the issue of stability, we formulate the discrete problem as follows:
Find $(\mathbf{u},p) \in \mathbf{V}_h^\Gamma \times Q_h^\Gamma$ such that with $a(\mathbf{u},\mathbf{v}) :=  \frac12 \sum_{i=1,2} \mu_i (D(\mathbf{u}),D(\mathbf{v}))_{\Omega_i},~\mathbf{u},\mathbf{v} \in \mathbf{V}_h^\Gamma,$ there holds
\begin{align*} 
a(\mathbf{u},\mathbf{v})
& - \sum_{i=1,2} \! (\mathrm{div}(\mathbf{v}),p)_{\Omega_i} 
- \sum_{i=1,2} \! (\mathrm{div}(\mathbf{u}),q)_{\Omega_i}
\!+ N((\mathbf{u},p),(\mathbf{v},q))
\!- J(p,q) \\
& = \!\sum_{i=1,2} \rho_i (\mathbf{g},\mathbf{v})_{\Omega_i} + f(\mathbf{v}),
\end{align*}
for all $(\mathbf{v},q) \in \mathbf{V}_h^\Gamma \times Q_h^\Gamma$. Here $(\cdot,\cdot)_{S}$ denotes the usual $L^2$ scalar product over the domain $S\in\{\Omega_1,\Omega_2,\Gamma\}$. The integrals over $\Omega_i,~i=1,2$ ensure consistency with respect to \eqref{eq:inner}. Consistency with respect to \eqref{eq:interface} has to be implemented through a suitable choice of the bilinear form $N((\cdot,\cdot),(\cdot,\cdot))$ and the linear form $f(\cdot)$. The additional stabilization bilinear form $J(\cdot,\cdot)$ is further introduced to ensure inf-sup-stabilty. Both aspects are discussed below.

\subsection{Unfitted Nitsche discretization to impose interface conditions}\label{sec:nitsche}
To implement the interface conditions, continuity of the velocity and conservation of momentum through the interface (in a weak sense) we consider Nitsche's method. 
We do this analogously to the Nitsche-XFEM for a scalar problem in \cite{hansbo2002unfitted}.
\begin{align*}
N((\mathbf{u},p),(\mathbf{v},q)) & := 
(\{\!\!\{ \sigma(\mathbf{u},p) \cdot \mathbf{n} \}\!\!\}, [\![ \mathbf{v} ]\!])_{\Gamma}
+ (\{\!\!\{ \sigma(\mathbf{v},q) \cdot \mathbf{n} \}\!\!\}, [\![ \mathbf{u} ]\!])_{\Gamma}
+ (\frac{\lambda}{h} \{\!\!\{ \mu \}\!\!\} [\![ \mathbf{u} ]\!], [\![ \mathbf{v} ]\!])_{\Gamma}, \\ f(\mathbf{v}) & := (\mathbf{f},\kappa_1 \mathbf{v}|_{\Omega_2} + \kappa_2 \mathbf{v}|_{\Omega_1})_{\Gamma}
\end{align*}
Here $\{\!\!\{v\}\!\!\} := \kappa_1 v|_{\Omega_1} + \kappa_2 v|_{\Omega_2}$, $\kappa_1+\kappa_2=1$, is a weighted average which plays an important role for the stability of the method. 
Together with $f(\cdot)$ the first term in $N(\cdot,\cdot)$ ensures consistency of the variational formulation and is derived by a reformulation of the terms stemming from partial integration. The second term is added for symmetry reasons, which is consistent due to $[\![ \mathbf{u} ]\!]=0$ on $\Gamma$ for the solution $\mathbf{u}$. The last term ensures coercivity of the viscosity operator for $\lambda$ sufficiently large and again vanishes for the solution $\mathbf{u}$. 
In view of the stability discussion, we define the bilinear forms
\begin{align*}
  A(\mathbf{u},\mathbf{v}) & := a(\mathbf{u},\mathbf{v}) + N((\mathbf{u},0),(\mathbf{v},0)), \\
  b(\mathbf{u},q) & := - {\textstyle \sum_{i=1,2} } (\mathrm{div}(\mathbf{u}),q)_{\Omega_i} + N((\mathbf{u},0),(0,q)), \\
k((\mathbf{u},p),(\mathbf{v},q)) & := A(\mathbf{u},\mathbf{v}) + b(\mathbf{u},q) 
+ b(\mathbf{v},p) - J(p,q),
\end{align*}
for $\mathbf{u},\mathbf{v} \in \mathbf{V}_h^\Gamma$ and $p,q \in Q_h^\Gamma$.
For the weighting $\kappa_1 = 0$ if $|T\cap\Omega_1|/|T| \leq \frac12$ and $\kappa_1 = 1$ otherwise the Nitsche formulation is known to be coercive, i.e. 
$\Vert \mathbf{u} \Vert_A := \sqrt{A(\mathbf{u},\mathbf{u})}, ~ \mathbf{u} \in \mathbf{V}_h^\Gamma$ defines a norm on $\mathbf{V}_h^\Gamma$,
cf. \cite[Lemma 5.1]{lehrenfeldreusken2016} and \cite[Lemma 3.5]{massjung12}. We note that for $\mathbf{u}, \mathbf{v} \in \mathbf{V}_h \subset [H^1(\Omega)]^d$ there holds $N((\mathbf{u},p),(\mathbf{v},q)) = 0$. Further for $\mathbf{u} \in \mathbf{V}_h^\Gamma$ and $q \in Q_h^\Gamma$ we have $k((\mathbf{u},p),(\mathbf{u},-p)) = A(\mathbf{u},\mathbf{u}) + J(p,p)$.

 \subsection{Inf-sup-stability and the ghost penalty stabilization}\label{sec:ghostpen}
One important aspect in the discretization of the Stokes problem is the design of LBB-stable velocity-pressure finite element spaces or the application of proper stabilization schemes.
In the context of unfitted finite element formulations this problem has been investigated in the literature for different velocity-pressure spaces: 

In \cite{zahedi} the space $\mathbf{V}_h^{\text{iso},\Gamma} \times Q_h^\Gamma$ is used with $\mathbf{V}_h^{\text{iso},\Gamma} = \mathbf{V}_h^{\text{iso}}|_{\Omega_1} + \mathbf{V}_h^{\text{iso}}|_{\Omega_2}$ where $\mathbf{V}_h^{\text{iso}}$ is the space of continuous piecewise linear functions on a once refined mesh, so that $\text{dim}(\mathbf{V}_h^{\text{iso}})=\text{dim}(\mathbf{V}_h)$. Inf-sup stability is shown for this velocity-pressure pair only with an additional stabilization term, the ``ghost penalty'' stabilization explained below. With this stabilization first order results for the $H^1$ norm error in the velocity are obtained.
In \cite{kirchhartetal2015} the ghost penalty stabilization has been used to prove inf-sup stability for the velocity-pressure pair $\mathbf{V}_h \times Q_h^\Gamma$.
In the recent paper \cite{Wang2015820} a stabilized equal-order space $[Q_h^\Gamma]^d \times Q_h^\Gamma$ has been combined with the ghost-penalty method to achieve a robust and first order (in the $H^1$ norm of the velocity) method.
We also mention the publication \cite{Cattaneo15} which considers (among others) the velocity space $\mathbf{V}_h^{\text{bub},\Gamma} = \mathbf{V}_h^{\text{bub}}|_{\Omega_1} \oplus \mathbf{V}_h^{\text{bub}}|_{\Omega_1}$ where 
$\mathbf{V}_h^{\text{bub}}$ is the space of continuous piecewise linear functions enriched with interior bubble functions.

In all these publications robust methods for Stokes interface problems have been derived using additional stabilizations, especially the ghost-penalty method. 
Additionally to provide inf-sup-stability independent of the interface position, these stabilizations add control on the conditioning of linear systems and thereby facilitate the treatment of arising linear systems with iterative methods.
Nevertheless, the above mentioned methods are -- in contrast to the method presented here -- not able to provide higher order accuracy. 
This is obvious as none of the above velocity spaces, $\mathbf{V}_h^{\text{iso},\Gamma}$, $\mathbf{V}_h$, $[Q_h^\Gamma]^d$ or $\mathbf{V}_h^{\text{bub},\Gamma}$ provide more than first order convergence (in the $H^1$ norm of the velocity) for velocities with weak discontinuities across the interface. Note that $\mathbf{V}_h^\Gamma$ provides these same higher order approximation error bounds, cf. \eqref{eq:secorderbound}.
In this sense the present contribution constitutes a step forward in the direction of higher order discretizations for Stokes interface problems on level set domains. 

To ensure stabilization we also apply the ghost penalty stabilization introduced in \cite{burman10,burman2012fictitious}.
Let (for $i = 1,2$)
\begin{equation*}
 \mathcal{F}_i^\Gamma := \{ F = T_a \cap T_b; \mathrm{meas}_{d-1} (F) > 0; T_a \neq T_b; T_a \cap \Omega_i \neq \emptyset, T_b  \cap \Omega_i \neq \emptyset ; T_a\text{ or } T_b \text{ are cut}\},
\end{equation*}
be the set of faces within the band of cut elements. On this set we add the stabilization bilinear form
\begin{equation*}
 J(p,q) :=  \gamma ~ {\textstyle \sum_{i=1,2} \sum_{F \in \mathcal{F}_i^\Gamma} } \mu_i^{-1} h_F^{3} (\jump{\partial_n \mathcal{E}_{i,h} p}, \jump{\partial_n \mathcal{E}_{i,h} q})_{F}, \quad p,q \in Q_h^\Gamma,
\end{equation*}
with $\mathcal{E}_{i,h}$ the canonical extension of discrete functions in $Q_h^\Gamma$ from $\Omega_i$ to $\Omega_i^+$ (the domain of all elements which have some part in $\Omega_i$), $h_F=\max\{h_{T_a},h_{T_b}\}$ where $F = T_a \cap T_b$ and $\gamma > 0$ the stabilization parameter. 
This additional bilinear form stabilizes the discrete formulation by penalizing discontinuities in the derivative across element faces which are close to the interface. For domainwise smooth solutions this stabilization is obviously consistent. For the discussion of inf-sup stability we introduce the following norm on $\mathbf{V}_h^\Gamma \times Q_h^\Gamma$:
$$
\enorm{(\mathbf{u},p)}^2 := A(\mathbf{u},\mathbf{u}) + { \textstyle \sum_{i=1,2} } \Vert \mu_i^{-\frac12} \mathcal{E}_{i,h} p\Vert_{L^2(\Omega_i^+)}^2 + J(p,p),\quad (\mathbf{u},p) \in \mathbf{V}_h^\Gamma \times Q_h^\Gamma. 
$$
With respect to this norm we are able to deduce an inf-sup result for the discretization with $\mathbf{V}_h^\Gamma \times Q_h^\Gamma$ and the ghost penalty stabilization. Key ingredient for this are the results obtained in \cite{kirchhartetal2015} for a discretization with $\mathbf{V}_h \times Q_h^\Gamma$.
\begin{theorem}
There exist $h_0, \lambda_0, \gamma_0, c_s>0$, such that for all $h < h_0, \lambda > \lambda_0, \gamma > \gamma_0$ there holds the inf-sup condition
\begin{equation*}
\sup_{(\mathbf{v},q) \in \mathbf{V}_h^\Gamma \times Q_h^\Gamma} \frac{k((\mathbf{u},p),(\mathbf{v},q))}{\enorm{(\mathbf{v},q)}} \geq c_s \enorm{(\mathbf{u},p)} \quad  \text{ for all } (\mathbf{u},p) \in \mathbf{V}_h^\Gamma \times Q_h^\Gamma.
\end{equation*}
In particular the constant $c_s>0$ does not depend on $h$ or the position of the interface $\Gamma$ relative to the mesh.
\end{theorem}
\begin{proof}[Sketch of the proof]
We fix $(\mathbf{u},p) \in \mathbf{V}_h^\Gamma \times Q_h^\Gamma$. 
The most important ingredient in the proof is \cite[Theorem 5.3]{kirchhartetal2015} which states that for given $p \in Q_h^\Gamma$ there exists a $\mathbf{w} \in \mathbf{V}_h$ such that for constants $c_1,c_2 >0$ independent of $h$ and $\Gamma$ there holds
$$
b(\mathbf{w},p) \geq c_1 
 { \textstyle \sum_{i=1,2} } \Vert \mu_i^{-\frac12} \mathcal{E}_{i,h} p\Vert_{L^2(\Omega_i^+)}^2 \!\! - c_2 J(p,p)\text{ and }
\Vert \mathbf{w} \Vert_A^2 = { \textstyle \sum_{i=1,2} } \Vert \mu_i^{-\frac12} \mathcal{E}_{i,h} p\Vert_{L^2(\Omega_i^+)}^2. 
$$
As $\mathbf{V}_h$ is a subspace of $\mathbf{V}_h^\Gamma$ this function $\mathbf{w}$ allows to control the pressure as in \cite{kirchhartetal2015}.
Analogously to the proof of \cite[Theorem 5.4]{kirchhartetal2015} we can take $(\mathbf{v},q)=(\mathbf{u}+\alpha\mathbf{w},-p)$ with a suitable choice for $\alpha$ to 
obtain
$
k((\mathbf{u},p),(\mathbf{v},q)) \geq c \enorm{(\mathbf{u},p)}
$ with constants $\alpha, c > 0$ which are independent of $h$ and $\Gamma$.
Combining this with $\enorm{(\mathbf{v},q)} \leq c^\ast(\alpha)\, \enorm{(\mathbf{u},p)}$ gives the result.
\end{proof}
Using standard techniques from the error analysis of non-conforming finite element methods optimal order a priori error bounds follow from this inf-sup result. 
Until now we assumed that numerical integration can be carried out exactly. In practice however one has to deal with approximations to the domains $\Omega_i,~i=1,2$ and the interface $\Gamma$. In order not to lose optimal order convergence we apply a new approach for the geometry approximation. This is discussed next.

\subsection{High order geometry approximation}\label{sec:geomapprox}
One major issue in the design and realization of high order unfitted finite element methods is the problem of numerical integration on domains which are only implicitly described by a level set function $\phi$. Integrals of the form $\int_S f \, dx$ have to be computed for $S\in \{\Omega_1,\Omega_2,\Gamma\}$, with $\Gamma=\{\phi=0\}$ and $\Omega_i := \{\phi \gtrless 0\}$. A standard technique is based on a linear interpolation $I_h\phi$ of $\phi$ which 
results in explicit and (only) second order accurate reconstructions $\Gamma^{\text{lin}}$ and $\Omega_i^{\text{lin}},~i=1,2$. 

\begin{figure}[h!]
  \vspace*{-0.2cm}
  \small
  \begin{center}
    \begin{tabular}{rc@{}c@{}c@{}c@{}c}
      &
        \begin{minipage}{0.25\textwidth}
          \includegraphics[width=0.925\textwidth]{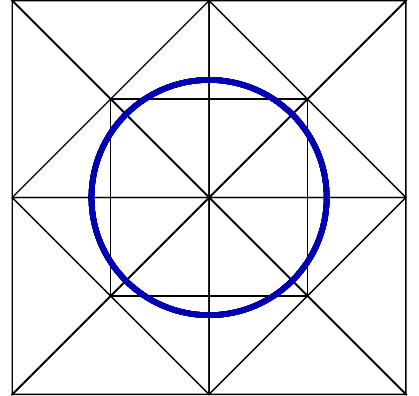} 
        \end{minipage}
      &
        \begin{minipage}{0.025\textwidth}
          \hspace*{-0.25cm} \Large +
        \end{minipage}
      &
        \begin{minipage}{0.25\textwidth}
          \includegraphics[width=0.925\textwidth]{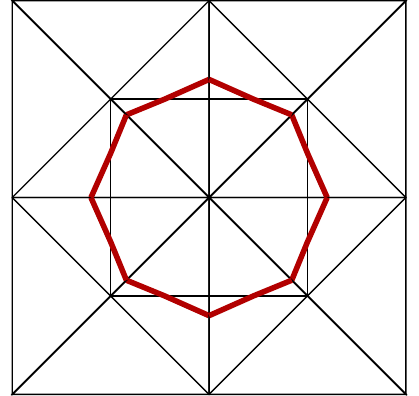} 
        \end{minipage}
      &
        \begin{minipage}{0.05\textwidth}
          $
          \displaystyle
          \stackrel{\displaystyle\mathbf{\Psi}_h}{\longrightarrow}
          $
        \end{minipage}
      &
        \begin{minipage}{0.25\textwidth}
          \includegraphics[width=0.925\textwidth]{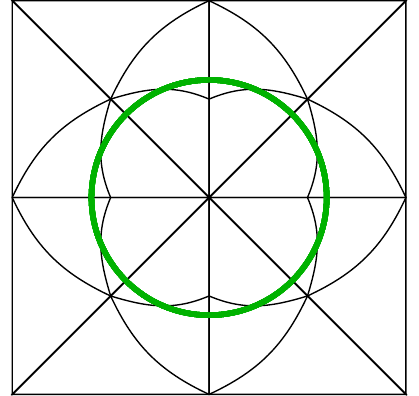} 
        \end{minipage}
        \\
      interface: \hspace*{-1cm} & $ \{\phi=0 \}$ && $ \Gamma^{\text{lin}} = \{I_h \phi = 0\}$ && $\mathbf{\Psi}_h( \{ I_h \phi = 0\} )$ \\
      mesh: \hspace*{-1cm} & $\mathcal{T}_h$ && $\mathcal{T}_h$ && $\mathbf{\Psi}_h(\mathcal{T}_h)$ \\
      accuracy: \hspace*{-1cm} & \underline{$\mathcal{O}(h^{k+1})$} && $\mathcal{O}(h^{2})$ && \underline{$\mathcal{O}(h^{k+1})$} \\
      representation: \hspace*{-1cm} & implicit && \underline{explicit} && \underline{explicit}
    \end{tabular}
  \end{center}
  \caption{Main idea of the method in \cite{lehrenfeld2016cmame}: The geometry description with the level set function $\phi$ is highly accurate but implicit (left). The zero level  $\Gamma^{\text{lin}}$ of the piecewise linear interpolation $I_h \phi$ has an  explicit representation but is only second order accurate (center). $\Gamma^{\text{lin}}$ is mapped towards the interface $\{\phi=0\}$ applying the mesh transformation $\mathbf{\Psi}_h$ resulting in a highly accurate and explicit representation (right).}
  \label{fig:idea} 
\end{figure}

In \cite{lehrenfeld2016cmame} a novel approach has been proposed to improve this by applying a parametric mapping $\mathbf{\Psi}_h: \Omega \rightarrow \Omega$, $\mathbf{\Psi}_h \in \mathbf{V}_h$ of the \emph{underlying mesh} such that $\phi \circ \mathbf{\Psi}_h \approx I_h \phi$. The representation of the resulting geometry is still explicit and thus allows for the application of fairly simple quadrature rules. We refer to Figure \ref{fig:idea} for a sketch, and to \cite{lehrenfeld2016cmame} for details on the construction of the mapping $\mathbf{\Psi}_h$.
To make use of this higher order geometry approximation, $\mathbf{\Psi}_h$ has to be considered also in the discretization rendering the resulting methods \emph{isoparametric} unfitted FE methods.
In the discretization above we have to replace $\Gamma$ with $\Gamma_h \!= \mathbf{\Psi}_h(\Gamma^{\text{lin}})$, $\Omega_i$ with $\Omega_{i,h} \!= \mathbf{\Psi}_h(\Omega_i^{\text{lin}})$, replace $\mathbf{V}_h^\Gamma$ with $\mathcal{V}_h^\Gamma := \{ v \circ \mathbf{\Psi}_h^{-1}\mid v \in \mathbf{V}_h^\Gamma\} = \{ v \mid v \circ \mathbf{\Psi}_h \in \mathbf{V}_h^\Gamma\}$ and $Q_h^\Gamma$ with $\mathcal{Q}_h^\Gamma := \{ q \mid q \circ \mathbf{\Psi}_h \in {Q}_h^\Gamma\}$.
In \cite{lehrenfeldreusken2016} rigorous high order error bounds have been derived for the discretization error (including the consideration of geometry errors) of an unfitted finite element discretization for a scalar unfitted interface problem. 

\section{Numerical example} \label{sec:numex}
We consider a numerical example from the literature, cf. \cite{kirchhartetal2015}, with the domain $\Omega=[-1,1]^2$ and an interface $\Gamma:=\{\mathbf{x}\in \Omega : \phi(\mathbf{x}) := \Vert \mathbf{x} \Vert_2 - r_\Gamma = 0\}$ where $r_\Gamma=2/3$. On this domain we solve the Stokes interface problem with $(\mu_1, \mu_2) = (1,10)$ and $\mathbf{f}_\Gamma = 1/2 \cdot \mathbf{n}_\Gamma$. The boundary data $\mathbf{u}_D$ and the force $\mathbf{g}$ are set such that the solution is:
\begin{align*}
\hspace*{-0.8cm}
  \mathbf{u}(\mathbf{x}) &= e^{-\Vert \mathbf{x} \Vert_2^2} (-x_2,x_1)^\top
                             \left\{ \begin{array}{l@{,\ }c} 
                                       \mu_1^{-1} & \Vert \mathbf{x} \Vert_2 \le r_\Gamma, \\
                                       \mu_2^{-1}\! + (\mu_1^{-1}\!\! -\! \mu_2^{-1})
                                       e^{\Vert \mathbf{x} \Vert_2^2 - r_\Gamma^2} & \Vert \mathbf{x} \Vert_2 > r_\Gamma,
                                     \end{array} \right. \\  
  p(\mathbf{x}) &= -\frac{\pi}{18} + \left\{ \begin{array}{l@{,\ }c} 
                              x_1^3 + 1/2 & \Vert \mathbf{x} \Vert_2 < r_\Gamma, \\
                              x_1^3 & \Vert \mathbf{x} \Vert_2 > r_\Gamma.
                            \end{array} \right.
\end{align*}
Note that $\mathbf{u}\cdot \mathbf{n}_\Gamma = 0$ on $\Gamma$, but the velocity has kinks and the pressure has jumps across the interface.

Starting from a shape regular unstructured mesh (230 triangles) which is not fitted to the interface we consecutively refine the mesh 6 times resulting in 7 levels $L \in \{0,..,6\}$. On each mesh we applied three discretizations where we switch between applying and not applying the isoparametric mapping $\mathbf{\Psi}_h$ and between the velocity spaces $\mathbf{V}_h$ ($\mathcal{V}_h$) and $\mathbf{V}_h^\Gamma$ ($\mathcal{V}_h^\Gamma$). In all cases we use the ghost penalty stabilization with $\gamma = 0.1$ and the Nitsche parameter $\lambda = 20$.
The computations were carried out with the add-on package \texttt{ngsxfem} to the finite element library \texttt{NGSolve} \cite{schoeberl2014cpp11}. Direct solvers have been used to solve the arising linear systems. 

Let $(\mathbf{u}_h,p_h)$ be the discrete solution of the previously discussed discretizations. 
In the Tables \ref{tab:results1}-\ref{tab:results3} the error measures
$ e_{(\mathbf{u},p)} = \Vert p - p_h \Vert_{L^2(\Omega)} + \Vert \mathbf{u} - \mathbf{u}_h \Vert_{H^1(\Omega_1^\ast \cup \Omega_2^\ast)}$ 
and
$ e_{\mathbf{u},L^2} = \Vert \mathbf{u} - \mathbf{u}_h \Vert_{L^2(\Omega)}$
and corresponding experimental orders of convergence (eoc) are depicted. 
Here, the domains $\Omega_i^\ast,~i=1,2$ are, depending on the application of the mesh transformation $\mathbf{\Psi}_h$, either $\Omega_i^\ast = \Omega_i^{\text{lin}}$ or $\Omega_i^\ast = \Omega_{i,h} = \mathbf{\Psi}_h(\Omega_i^{\text{lin}}),~i=1,2$. 

We observe that the velocity enrichment is crucial to obtain good results. This is not surprising considering the sharp estimates in \eqref{eq:suboptest}. Applying the velocity enrichment without the isoparametric mapping still gives suboptimal results. This is due to the insufficient accuracy with respect to the geometry. The combination of both, the velocity enrichment and the isoparametric mapping, resolves this problem and optimal order convergence can be observed in both measures. 
\begin{table}[h!]
\begin{center}
\begin{tabular}{c
  c@{\hspace*{0.2cm}(}c@{)\hspace*{0.3cm}}
  c@{\hspace*{0.2cm}(}c@{)\hspace*{0.1cm}} 
}
\toprule
& \multicolumn{4}{c}{$\mathcal{V}_h \times \mathcal{Q}_h^\Gamma$} \\
\midrule
$L$ 
& $e_{(\mathbf{u},p)}$ & eoc & $e_{\mathbf{u},L^2}$ & eoc \\
\midrule
0 & \num{1.764192e-01} &  --- & \num{1.076343e-02} &  --- \\
1 & \num{1.355452e-01} &  0.4 & \num{5.457021e-03} &  1.0 \\
2 & \num{1.033889e-01} &  0.4 & \num{3.054609e-03} &  0.8 \\
3 & \num{7.024931e-02} &  0.6 & \num{1.607201e-03} &  0.9 \\
4 & \num{4.844310e-02} &  0.5 & \num{8.559501e-04} &  0.9 \\
5 & \num{3.350455e-02} &  0.5 & \num{3.973392e-04} &  1.1 \\
6 & \num{2.350496e-02} &  0.5 & \num{2.037749e-04} &  1.0 \\
\bottomrule
\end{tabular}
\end{center}
\caption{Convergence history for discretizations without velocity enrichment, but parametric transformation.}
\label{tab:results1}
\end{table}
\begin{table}[h!]
\begin{center}
\begin{tabular}{c
  c@{\hspace*{0.2cm}(}c@{)\hspace*{0.3cm}}
  c@{\hspace*{0.2cm}(}c@{)\hspace*{0.1cm}} 
}
\toprule
& \multicolumn{4}{c}{$\mathbf{V}_h^\Gamma \times \mathbf{Q}_h^\Gamma$} \\
\midrule
$L$ 
& $e_{(\mathbf{u},p)}$ & eoc & $e_{\mathbf{u},L^2}$ & eoc 
\\
\midrule
0 & \num{3.687400e-02} & --- & \num{4.835893e-04} & --- \\
1 & \num{1.415700e-02} & 1.4 & \num{1.792713e-04} & 1.4 \\
2 & \num{4.303471e-03} & 1.7 & \num{2.908379e-05} & 2.6 \\
3 & \num{1.313515e-03} & 1.7 & \num{4.188722e-06} & 2.8 \\
4 & \num{4.270432e-04} & 1.6 & \num{8.370972e-07} & 2.3 \\
5 & \num{1.498958e-04} & 1.5 & \num{1.597825e-07} & 2.4 \\
6 & \num{5.099163e-05} & 1.6 & \num{2.441986e-08} & 2.7 \\
\bottomrule
\end{tabular}
\end{center}
\caption{Convergence history for discretizations with velocity enrichment, but no parametric mapping.}
\label{tab:results2}
\end{table}

\begin{table}[h!]
\begin{center}
\begin{tabular}{c
  c@{\hspace*{0.2cm}(}c@{)\hspace*{0.3cm}}
  c@{\hspace*{0.2cm}(}c@{)\hspace*{0.1cm}} 
}
\toprule
& \multicolumn{4}{c}{$\mathcal{V}_h^\Gamma \times \mathcal{Q}_h^\Gamma$} \\
\midrule
$L$ 
& $e_{(\mathbf{u},p)}$ & eoc & $e_{\mathbf{u},L^2}$ & eoc 
\\
\midrule
0 & \num{3.193850e-02} & --- & \num{3.570931e-04} & ---  \\
1 & \num{9.282097e-03} & 1.8 & \num{7.257297e-05} & 2.3 \\
2 & \num{2.227819e-03} & 2.0 & \num{1.016501e-05} & 2.8 \\
3 & \num{5.460477e-04} & 2.0 & \num{1.319041e-06} & 3.0 \\
4 & \num{1.357255e-04} & 2.0 & \num{1.680337e-07} & 3.0 \\
5 & \num{3.382225e-05} & 2.0 & \num{2.119221e-08} & 3.0 \\
6 & \num{8.441304e-06} & 2.0 & \num{2.660372e-09} & 3.0 \\
\bottomrule
\end{tabular}
\end{center}
\caption{Convergence history for discretizations with velocity enrichment and parametric mapping.}
\label{tab:results3}
\end{table}

In further numerical studies we observed that - although the ghost penalty stabilization is necessary to prove the inf-sup stability in section \ref{sec:ghostpen} - we obtain almost identical results if we do not apply the ghost penalty stabilization ($\gamma=0$).

\bibliographystyle{alpha}
\bibliography{library}

\end{document}